\documentclass[11pt]{amsart}
\usepackage[margin=1.2in,marginparsep=0.1in,marginparwidth=1in]{geometry}
\usepackage{amssymb,amsmath,amsthm,amstext,amscd,latexsym,graphics,graphicx,bbm,caption}
\usepackage[usenames,dvipsnames,svgnames,table]{xcolor}
\usepackage[plainpages=false,colorlinks=true, pagebackref]{hyperref}
\usepackage{tikz-cd}
\usetikzlibrary{positioning}
\usepackage{multicol}
\usepackage{lipsum}

\tikzset{>=stealth}
\hypersetup{citecolor=Sepia,linkcolor=blue, urlcolor=blue}

\usepackage{cite}   

\makeatletter
\def\@tocline#1#2#3#4#5#6#7{\relax
  \ifnum #1>\c@tocdepth 
  \else
    \par \addpenalty\@secpenalty\addvspace{#2}%
    \begingroup \hyphenpenalty\@M
    \@ifempty{#4}{%
      \@tempdima\csname r@tocindent\number#1\endcsname\relax
    }{%
      \@tempdima#4\relax
    }%
    \parindent\z@ \leftskip#3\relax \advance\leftskip\@tempdima\relax
    \rightskip\@pnumwidth plus4em \parfillskip-\@pnumwidth
    #5\leavevmode\hskip-\@tempdima
      \ifcase #1
       \or\or \hskip 2em \or \hskip 2em \else \hskip 3em \fi%
      #6\nobreak\relax
    \dotfill\hbox to\@pnumwidth{\@tocpagenum{#7}}\par
    \nobreak
    \endgroup
  \fi}
\makeatother

\newtheorem{intro-thm}{Theorem}[]
\theoremstyle{plain}
\newtheorem{thm}{Theorem}[section]
\newtheorem{theorem}[thm]{Theorem}
\newtheorem{question}[thm]{Question}
\newtheorem{lemma}[thm]{Lemma}
\newtheorem{corollary}[thm]{Corollary}
\newtheorem{proposition}[thm]{Proposition}
\newtheorem{conjecture}[thm]{Conjecture}
\theoremstyle{definition}
\newtheorem{remark}[thm]{Remark}

\newtheorem{definition}[thm]{Definition}
\newtheorem{examples}[thm]{Examples}

\newcommand{\ilim}{\mathop{\varprojlim}\limits} 

\newcommand{\Union}{\bigcup}

\newcommand{\union}{\cup}

\renewcommand{\tilde}{\widetilde}
\newcommand{\sat}{{\rm sat}}

\newcommand{\sR}{{\mathcal R}}

\newcommand{\N}{{\mathbb N}}

\newcommand{\R}{{\mathbb R}}

\newcommand{\Z}{{\mathbb Z}}

\newcommand{\ab}[1]{\langle {#1} \rangle}
\newcommand{\ac}[1]{\{#1\}}

\newcommand{\Rings}{{\mathsf{Rings}}}
\newcommand{\Com}{{\mathsf{ComRings}}}
\newcommand{\Ab}{{\mathsf{Ab}}}

\input{xy}
\xyoption{all}

\begin{document}

\title{A universal construction of $p$-typical Witt vectors of associative rings} 

\author[S. Pisolkar]{Supriya Pisolkar} \address{ Indian Institute of
Science, Education and Research (IISER),  Homi Bhabha Road, Pashan,
Pune - 411008, India} \email{supriya@iiserpune.ac.in} 

\author[B. Samanta] {Biswanath Samanta} \address{ Indian Institute of
Science, Education and Research (IISER),  Homi Bhabha Road, Pashan,
Pune - 411008, India} \email{biswanath.samanta@students.iiserpune.ac.in} 

\thanks{The second-named author was supported by Ph.D. fellowship (File No: 09/936(0315)/2021-
EMR-I) of Council of Scientific \& Industrial Research (CSIR), India.} 
\maketitle 
\begin{abstract}
For a prime $p$ and an associative ring $R$ with unity, there are various constructions of $p$-typical Witt vectors of $R$, all of which specialize to the classical $p$-typical Witt vectors when $R$ is commutative. These constructions are endowed with a Verschiebung operator $V$ and a Teichm\"uller map $\ab{\cdot}$, and they satisfy the property that the map $
x \mapsto V\ab{x^p} - p\ab{x}$ is additive. In this paper, we adapt the group-theoretic universal characterization of classical $p$-typical Witt vectors proposed in \cite{ps} to the non-commutative setting. Our main result is that this approach yields a construction of Witt vectors for associative rings, denoted $E$, which specializes correctly to the classical Witt functor in the commutative case. The construction of $E$ is inspired by the Witt functor of Cuntz--Deninger (see \cite{cd}), and we show that $E$ is a universal pre-Witt functor (see Theorem~\ref{main2}), subject to an explicit conjecture concerning non-commutative polynomials (see Conjecture \ref{zind}). We further introduce the notion of a Witt functor and construct a universal Witt functor $\widehat{E}$, which is closely related to Hesselholt’s Witt functor $W_H$ (see \cite{h1}, \cite{h2}). We suspect that $W_H$ is, in fact, the universal Morita-invariant Witt functor.
\end{abstract}
\section{Introduction}

\noindent We fix a prime $p$ throughout this article. We will often refer to $p$-typical Witt vectors simply as Witt vectors. Let $\Com$, $\Rings$ and $\Ab$ denote respectively the category of commutative rings with unity, associative rings with unity and abelian groups. We have a classical $p$-typical Witt functor (see \cite{w}) $$W: \Com \to \Ab$$ There exist functorial maps 
\begin{enumerate}\label{W-properties}
    \item Verschiebung: $V:W(R)\to W(R)$ (group homomorphism)
    \item Teichm\"{u}ller: $\ab{\ }: R\to W(R)$ (a set map)
\end{enumerate}
which satisfy the following properties (see \cite{ps})
\begin{enumerate}
    \item $\ab{0}=0$. If $p\neq 2$ then $\ab{-x} = -\ab{x}$ for all $x \in R$.
    \item $x \mapsto V\ab{x^p} -p\ab{x}$ is an additive map from $R \to W(R)$. 
    \item $W(R)$ is complete w.r.t to the filtration $\{V^nW(R)\ | \ n \in \N_0 \}$.
    \item $R$ is $p$-torsion free $\implies W(R)$ is $p$-torsion free.
\end{enumerate}
\noindent The above properties were the basis of the definition of a pre-Witt functor in \cite{ps}. We recall the following characterization of $W$ 
\begin{theorem}\cite[Theorem 1.6]{ps}\label{commain} For $p\neq 2$, the classical functor of $p$-typical Witt vectors $W: \Com \to Ab$ is a universal pre-Witt functor.
\end{theorem}
\noindent The key point of the above characterization is that it uses only the group structure on $W(R)$, and not the ring structure. This is important because many known constructions of Witt vectors for non-commutative rings (see \cite{b}, \cite{J}) produce only abelian groups and are not equipped with a ring structure. A natural question, which we address in this paper, is the following: what is the universal Witt functor on the category $\Rings$ obtained by mimicking these properties?
This motivates the following definition which generalizes the definition of a pre-Witt functor \cite[Definition 1.3]{ps} to the case of associative rings (possibly non-commutative). 

\begin{definition}[pre-Witt functor]\label{def:prewitt}
    A pre-Witt functor is a functor $F:\Rings \to \Ab$ which has functorial maps 
\begin{enumerate}
    \item Verschiebung operator: $V: F(R)\to F(R)$ (group homomorphism)
    \item Teichm\"{u}ller map : $\ab{\ }: R\to F(R)$ (a set map)
\end{enumerate}
which satisfy the following properties 
\begin{enumerate}\label{properties}
    \item $\ab{0}=0$. If $p\neq 2$ then $\ab{-x} = -\ab{x}$ for all $x \in R$.
    \item $x \mapsto V\ab{x^p} -p\ab{x}$ is an additive map from $R \to F(R)$. 
    \item $F(R)$ is complete w.r.t to the filtration $\{V^nF(R)\ | \ n \in \N_0 \}$.
    \item $R$ and $\bar{R}$ are $p$-torsion free $\implies F(R)$ is $p$-torsion free, where $\bar{R} := \frac{R}{\ab{[R,R]}}$ and  $\ab{[R,R]}$ is the ideal generated by the commutator subgroup $[R,R]$. 
\end{enumerate}
\end{definition}

\noindent The two main examples of pre-Witt functors considered in this paper are
\begin{enumerate}
\item The functor $E : \Rings \to \Ab$ which is obtained by adapting the construction of Cuntz and Deninger (see Definition \ref{def:e})
\item The Witt functor $W_H: \Rings \to \Ab$ defined by Hesselholt (see \cite{h1}, \cite{h2}). 
\end{enumerate}
\vspace{2mm}
The following is the main result of this paper
\begin{theorem} \label{main} 
$E$ is a pre-Witt functor. Moreover 
\begin{enumerate}
\item The restriction of $E$ to $\Com$ is canonically isomorphic to the classical Witt functor $W$.
\item There exists a unique natural transformation $E \to W_H$, which is surjective.
\end{enumerate}
\end{theorem}

\begin{remark}
Here, we note that property~\eqref{properties}(4) has been slightly modified compared to property~(4) satisfied by $W$ in the commutative case. The main observation of this paper is that this modification requiring both $R$ and $\bar{R}$ to be $p$-torsion-free is necessary in order to show that $E$ restricts to the classical Witt functor $W$ on $\Com$. Moreover, since for a commutative ring $R$ the ring $\bar{R}$ is trivial (and hence vacuously torsion-free), this property specializes correctly to the commutative case. In addition, Hesselholt’s Witt functor $W_H$ satisfies these properties and is therefore a pre-Witt functor.
\end{remark}

\noindent We also show that modulo an explicit conjecture (see Conjecture \ref{zind}) related to non-commutative polynomials, $E$ is in fact a universal pre-Witt functor. 
\begin{theorem}\label{main2}
Let $p \neq 2$. Assuming Conjecture \ref{zind}, $E$ is a universal pre-Witt functor.  
\end{theorem}

\noindent One of the main motivations for this paper and also \cite{ps} was the following question 

\begin{question} \label{question:w_h} Can $W_H$ be characterized by universal properties? 
\end{question}

\noindent For a commutative ring $R$, the group of $p$-typical Witt vectors $W(R)$ has an additional property, i.e. that sum and difference of elements in $W(R)$ is given by explicit Witt polynomials. This implies the existence of a canonical bijection $R^{\N_0} \to W(R)$ of sets. A similar statement also holds for $W_H(R)$ for a non-commutative ring $R$, except that instead of a bijection, we have an epimorphism $R^{\N_0}\to W_H(R)$ (see \cite[Page 56]{h2}). Motivated by this, we call a pre-Witt functor $F$ a Witt functor if the sum and difference of elements in $F(R)$ are given by non-commutative Witt polynomials (see Definition \ref{rnen}). \\

\vspace{1mm}
\noindent It is natural to ask if there is a universal Witt functor on $\Rings$. To answer this, we define a Witt functor $\hat{E}$ that is a suitable quotient of the pre-Witt functor $E$ by a closed subgroup defined by the certain Witt polynomials (see Definition \ref{defn:ehat}). We have the following conjectural characterization of the functor $\hat{E}$.

\begin{theorem}\label{ehatuniversal}
The Conjecture \ref{zind} implies that  $\hat{E}$ is a universal Witt functor for $p\neq 2$. In particular there is a unique natural transformation from $\hat{E}\to W_H$ 
\end{theorem} 

\noindent One can show that $\hat{E}$ is not isomorphic to $W_H$ and moreover it is not Morita invariant. We would like to pose the following question related to \ref{question:w_h}.
\begin{question}
Is  $W_H$ is a universal Morita-invariant Witt Functor?
\end{question}
\vspace{2mm}

\noindent Theorems \ref{main2} and \ref{ehatuniversal} in this paper rely on the following conjecture concerning non-commutative polynomials. This conjecture is a direct analogue of \cite[Lemma 4.2]{ps}, which played a crucial role in the proof of the main result \cite[Theorem 1.6]{ps}. The statement below uses the notation introduced in Section \ref{section:e}.

\begin{conjecture}\label{zind}
Let $p\neq 2$. Let $A = \Z\{S\}$ be a non-commutative polynomial ring over a set $S$. Let $\{f_i\}_{i=1}^r$ be a finite set of distinct non-zero elements of $A$. Further assume  that $f_i\neq -f_j$ for any $i\neq j$. Then the subset $\{\ab{f_i}\}_{i=1}^r$ of $X(A)$ is $\Z$-linearly independent. 
\end{conjecture}

The outline of the paper is as follows. In the Section 2 we give the definition of the functor $E: \Rings \to \Ab$ and prove the first part of Theorem \ref{main} i.e we prove that $E$ is a pre-Witt functor and part(1) of Theorem \ref{main}. In Section 3, we prove Theorem \ref{main2}. In Section 4 we define the notion of a Witt functor and give the definition of the functor $\hat{E}$ and prove Theorem \ref{main}(2) witand Theorem \ref{ehatuniversal}. In Section \ref{evidence} we discuss evidence for the Conjecture \ref{zind}. 


\section{Definition and properties of $E: \Rings \to \Ab$} \label{section:e}

In this section we define the functor $E$ (see Definition \ref{def:e}) used in Theorem \ref{main}. We adapt the construction of $p$-typical Witt vectors by Cuntz and Deningers (see \cite{cd}) to enforce properties \ref{properties}.  We begin by recalling the definitions and results about the functor $X : \Rings \to \Ab$ as in \cite{cd}. Let $\N_0 := \N\union \{0\}$.

\begin{definition}[$X(R)$]\label{def:ncxa} Let $R \in \Rings$. We have
\begin{enumerate}
\item (Verschiebung) $V: R^{\N_0}\to R^{\N_0}$ given by 
$ V(r_0,r_1,...) := p(0, r_0,r_1,...)$.
\item (Teichm\"{u}ller map) For $r\in R$ define 
$ \ab{r}:= (r,r^p,r^{p^2},...)$.
\item $\tilde{X}(R)$ be the subgroup generated by $\{V^n\ab{r} \ | \ n\in \N_0, \ r\in R\}$. Define a topology on $R^{\N_0}$ using the sequence of decreasing subgroups  $\{V^n\tilde{X}(R) \}_{n\in\N}$.  
\item Define $X(R)$ as the closure of $\tilde{X}(R)$ in $R^{\N_0}$ w.r.t. this topology. Note that $X(R)$ is invariant under $V$.
\end{enumerate}
\end{definition}

\noindent{Properties of $X(R)$ :}\label{prop:ncxa}
\begin{enumerate}
\item $X: \Rings \to \Ab$ is a functor. 
\item For a non-commutative ring $R$, the topology on $X(R)$ induced by $\{V^n \tilde{X}(R)\}$ is finer than the one coming from product topology on $\R^{\N \cup \{0\}}$. Indeed, let $A:=\Z\ac{x, y, z}$. Then the element $\ab{x+y}+ \ab{-x}+ \ab{-y}+V\ab{z}$ is not in $V^n(\tilde{X}(A))$ for any $n \geq 1$. Hence, $(\{0\} \times A^{\N}) \cap \tilde{X}(A)$ is not open in the $V$-topology.
\item For $R \in \Rings$, $X(R) \cong \ilim_n \frac{X(R)}{V^n(X(R))}$. Thus, $X(R)$ has $V$-topology and 
every element of $X(R)$ can be written as $\sum_{n_i}V^{n_i}\ab{a_{n_i}}$ where for $M \in \N$, there will only be finitely many $n_i \leq M$ counted with repetitions.
\end{enumerate}

\begin{definition}[$X_I$ and $X_I^{\sat}$]\label{def:XI}
For $R \in \Rings$, consider the free presentation $$0 \to I \to A \to R \to 0 $$ where $A:=\Z \{ R\}$ be the non-commutative polynomial ring on the set $R$.  Let $X_I(A) \subset X(A)$ be the closed subgroup generated by $$\{V^n\ab{a}-V^n\ab{b} \ | \ n\geq 0,  \ a,b\in A \ \text{ such that } a-b\in I\}.$$ We will denote $X_I$ as a short for $X_I(A)$. We denote the $p$-saturation of $X_I$ by 
$$X_I^{\sat} := \{ \alpha \in X(A) \mid \ p^{\ell}\alpha \in X_I  \text{ for some } \ell \in \N_0 \} $$  
\end{definition}
\noindent It is easy to see that $X_I \subseteq X_I^{\sat}$.
\begin{remark}
For $R \in \Com$, the closed subgroup $X(I)$ generated by $\{V^n\ab{a} \ | \ n \in \N_0,  \ a\in  I\}$ as given in \cite[Prop.1.2]{cd} coincides with the closed subgroup $X_I(A)$ generated by 
$$\{V^n\ab{a}-V^n\ab{b} \ |a,\ b \in A \text{ such that } a-b \in I,  n \in \N_0 \}.$$  The subgroups $X(I)$ and $X_I(A)$ will not coincide for a non-commutative ring. For example, take a non-commutative polynomial ring $A:= \Z\{X, Y\}$ and $I:=\ab{X, Y} $ then $\ab{X} - \ab{Y}$ is in $X_I(A)$ but not $X(I)$. This is the reason why we need to modify the definition of $E$ so that it works in the non-commutative set-up. 
\end{remark}
\vspace{2mm}

\begin{definition}[$E(R)$]\label{def:e}
For $R \in \Rings$ we denote by $\bar{R}$ the quotient ring $\frac{R}{\ab{[R,R]}}$ where $\ab{[R,R]}$ is the ideal generated by the commutator subgroup $[R,R]$. Suppose there exists a morphism $\phi: S \to R$ in $\Rings$ where $S$, $\bar{S}$ are $p$-torsion free. We consider the free presentations of $R$ and $S$ as follows. 

\[
\begin{tikzcd}[column sep=small, row sep=1.5em]
0 \arrow[r] & J \arrow[d, dashed] \arrow[r] & \Z\{S\} \arrow[d, dashed,"\tilde{\phi}"] \arrow[r, "\pi_2"] & S \arrow[d, "\phi"] \arrow[r] & 0 \\
0 \arrow[r] & I \arrow[r] & \Z\{R\} \arrow[r, "\pi_1"] & R \arrow[r] & 0
\end{tikzcd}
\tag{1}
\] 

Here, $\tilde{\phi}$ is a lift of $\phi$. Applying the functor $X$ we get the following diagram.
\[
\begin{tikzcd}[column sep=small, row sep=1.5em]
 & X_J \arrow[d, dashed] \arrow[r] & X(\Z\{S\}) \arrow[d, dashed,"\tilde{\phi}"] \arrow[r] & X(S) \arrow[d, "\phi"] \arrow[r] & 0 \\
 & X_I \arrow[r] & X(\Z\{R\}) \arrow[r] & X(R) \arrow[r] & 0
\end{tikzcd}
\tag{2}
\] 

Following the notation as in the Definition \ref{def:XI}, we define 
$$\tilde{X_I} : = \text{ the closed subgroup of } X(\Z\{R\}) \text{ generated by  } \{X_I, \Sigma_I\}$$
where for a lift $\tilde{\phi}$ as above
\[
\begin{aligned}
\Sigma_I \;:=\; \bigl\{\, \alpha \in X(\Z\{R\}) \ \bigm|\;
&  \alpha = \tilde{\phi}(\beta) \text{ for some } \beta \in X_J^{\mathrm{sat}}
\bigr\}.
\end{aligned}
\]
We leave it to the reader to check that the set $\Sigma_I$ is independent of the lift $\tilde{\phi}$.
Finally, we define $$E(R):= \frac{X(\Z\{R\})}{\tilde{X_I}}$$ 
\end{definition}

\begin{remark}\label{rmk:e} We would like to note the following observations regarding $E(R)$ which directly follow from the definition. Some of these are technical in nature however they are needed later. 
\begin{enumerate}
\item The group $E(R)$ is naturally endowed with an endomorphism $V: E(R) \to E(R)$ and a set map $\ab{ \ }: R \to E(R)$ (both induced from that on $X(R)$). Here, $\ab{r} := \ab{[r]}$, where $[r] \in \Z\{R\}$ is a lift of $r$. It is easy to see from the definition of $X_I$ that, $\ab{[r]} (\text{ modulo } X_I)$  is independent of choice of a lift of $r$ in $\Z\{R\}$. 
\item For $R \in \Rings$, $X_I \subseteq \tilde{X_I} \subseteq X_I^{\sat}$. Indeed, for $\alpha \in \tilde{X_I}$, there exists $\beta \in X_J^{\sat}$ such that $\tilde{\phi}(\beta) = \alpha$. As $\beta \in X_J^{\sat}$, $p^r\beta \in X_J$ for some $r \in \N_0$. Since $\tilde{\phi}(X_J) \subseteq X_I$ we get that $p^r\alpha \in X_I$. Thus $\alpha \in X_I^{\sat}$.
\item For $R \in \Rings$, if $R$ and $\bar{R}$ both are $p$-torsion free then $\tilde{X_I} = X_I^{\sat}$. In particular, $E(R)$ is $p$-torsion free. To prove this, it is now enough to show that $\tilde{X_I} \supseteq X_I^{\sat}$. This is straightforward by taking $S = R$ in the definition of $\Sigma_I$.
\end{enumerate}
\end{remark}
\noindent We will now prove the first part of the Theorem \ref{main}. We will show that $E$ is a pre-Witt functor. Before showing this, we will prove the following important Lemma. The idea of proof of this lemma will be used repetitively later in this paper.\\

\vspace{2mm}
\noindent Let $R\in \Rings$. In the definition of $E(R)$ (see Definition \ref{def:e}), we chose the canonical free presentation $0 \to I \to \Z\{R\} \to R\to 0$ in order to define $E(R):= X(\Z\{R\})/\tilde{X_I}$. However if  $$ 0 \to K \to B \to R \to 0$$ is any other free presentation we can also define the group 
$$E_B(R) := \frac{X(B)}{\tilde{X_K}}$$
where $\tilde{X_K}$ is  defined similarly. 
We now show that 
\begin{lemma}\label{ind:e} $E(R) \cong E_B(R)$.
\end{lemma}
\begin{proof} Suppose that we have following two free presentations of $R$.
 $$0 \to I \to A:=\Z\{R\} \to R \to 0$$
 $$0 \to K \to B \to R \to 0 $$ 
We need to show that $E(R) \cong E_B(R)$. We prove this result in following steps.\\

\noindent {\bf Step(1) :} In this step we will show that there exists a group homomorphism $\tilde{\psi}: X(A) \to X(B)$, such that $\tilde{\psi}(\tilde{X_I}) \subseteq \tilde{\psi}(\tilde{X_K})$. Hence it induces a well-defined morphism $\tilde{\psi}: E(R) \to E_B(R)$. \\

\noindent Consider the set maps $ R \stackrel{[\cdot]_1}{\longrightarrow}  A$ and $R \stackrel{[\cdot]_2}{\longrightarrow} B$. Define $\tilde{\psi}: X(A) \to X(B)$ as $$\psi(V^n\ab{[r]_1}) := V^n\ab{[r]_2}.$$ This induces a well defined group homomorphism $\tilde{\psi}: X(A)/X_I \to X(B)$.  It remains to show that $\tilde{\psi}(\tilde{X_I} ) \subseteq \tilde{X_K}$. It is enough to show that $\tilde{\psi}(\Sigma_I ) \subseteq \Sigma_K$. \\

\noindent Suppose there exists a morphism $\phi: S \to R$ where $S$ and $\bar{S}$ are $p$-torsion free. Consider the below diagrams similar to the one in Definition \ref{def:e}. \\

\[\begin{tikzcd} \label{diag1}
0 \arrow[r] & J \arrow[d, dashed] \arrow[r] & \Z\{S\} \arrow[d, "\tilde{\phi}", dashed] \arrow[r] & S \arrow[d, "\phi"] \arrow[r] & 0 \\
0 \arrow[r] & I \arrow[r] \arrow[d]  & A \arrow[r]\arrow[d, "\tilde{\psi}"]           & R \arrow[r]\arrow[d, "id"]          & 0 \\
0 \arrow[r] & K \arrow[r]           & B \arrow[r]           & R \arrow[r]           & 0
\end{tikzcd}\]

\noindent We get the following induced diagram,
\[\begin{tikzcd} \label{diag2}
 X_J \arrow[d, dashed] \arrow[r] & X(\Z\{S\}) \arrow[d, "\tilde{\phi}", dashed] \arrow[r] & X(S) \arrow[d, "\phi"] \arrow[r] & 0 \\
 X_I \arrow[r] \arrow[d]  & X(A) \arrow[r]\arrow[d, "\tilde{\psi}"]           & X(R) \arrow[r]\arrow[d, "id"]          & 0 \\
 X_K \arrow[r]           & X(B) \arrow[r]           & X(R) \arrow[r]           & 0
\end{tikzcd}\]

\noindent For $\alpha \in \Sigma_I$, $\alpha = \tilde{\phi}(\beta)$ for some $\beta \in X_J^{\sat}$. Thus $\tilde{\psi}(\alpha) = \tilde{\psi} \circ \tilde{\phi}(\beta)$ and hence $\tilde{\psi}(\alpha) \in \Sigma_K$. Thus we have a well defined homomorphism which we again denote as, 
$$\tilde{\psi} : E(R) \to E_B(R)$$

\noindent {\bf Step 2:} We will now show that $\tilde{\psi}: E(R) \to E_B(R)$ is an isomorphism. To prove this, we will establish a group homomorphism $\chi: E_B(R) \to E(R)$ and show that $\tilde{\psi} \circ \chi = id_{E_B(R)}$ and $\chi \circ \psi = id_{E(R)}$. Since $A$ and $B$ both are free rings,  we will reverse the roles of $A$ and $B$ in the step(1), to get a group homomorphism $\chi: E_B(R) \to E(R)$.  For $b \in \Z\{B\}$,

$$\tilde{\psi} \circ \chi(V^n\ab{b}) = V^n\ab{\tilde{\psi} \circ \chi(b)} \equiv V^n\ab{b} ( \text{mod } X_I) \text{ as }  \tilde{\psi} \circ \chi(b)-b \in I $$
\end{proof}

\noindent The following result, which is consequence of the above lemma, plays an important role in the proof of Theorem \ref{main}. 

\begin{lemma}\label{EisX}For a non-commutative free polynomial algebra $A$, $E(A) \cong X(A)$.
\end{lemma}
\begin{proof} This follows by taking the free presentation $0 \to 0 \to A \xrightarrow{{\rm id}} A \to 0$ and using the Lemma \ref{ind:e} 
\end{proof}

\begin{theorem}\label{e:pre-witt}$E: \Rings \to \Ab$ is a pre-Witt functor.
\end{theorem}
\begin{proof} We first observer that $E$ is a functor. Let $f : R_1 \to R_2$ be a morphism in $\Rings$. Consider the free presentations of $R_1$ and $R_2$,
\[\begin{tikzcd}
0 \arrow[r] & I \arrow[r] \arrow[d]  & \Z\{R_1\} \arrow[r]\arrow[d, "\psi"]           & R \arrow[r]\arrow[d, "f"]          & 0 \\
0 \arrow[r] & K \arrow[r]           & \Z\{R_2\} \arrow[r]           & R \arrow[r]           & 0 
\end{tikzcd}\]
Following similar arguments as in the proof of the Step(1) of Lemma \ref{ind:e} we get a group homomorphism $X(Z\{R_1\}) \to X(\Z\{R_2\})/\tilde{X_K}$ which factors via $E(R_1)$. \\

In order to show that $E$ is a pre-Witt functor we need to prove the four properties stated in \ref{properties}. The first two properties are easily seen to be satisfied by $E(R)$ since they are satisfied by $X(\Z\{R\})$. For the third property, we know that for $R \in \Rings$, $X(\Z\{R\})$ is $V$-complete. Moreover, $\tilde{X_I}$ is a closed subgroup of $X(\Z\{R\})$ where $X_I$ and $\Sigma_I$ are $V$-invariant. Thus $E(R)$ is $V$-complete. For the fourth property, assume that $R$ and $\bar{R}$ are $p$-torsion free. By the Remark \ref{rmk:e}(3), $\tilde{X_I} = X_I^{\sat}$, hence $E(R)$ is $p$-torsion free. Thus $E$ is a pre-Witt functor. 
\end{proof}

\noindent The remaining section is devoted to proving the following theorem which shows that the functor $E$ when restricted to commutative rings, gives the classical $p$-typical Witt functor.  
\begin{theorem}\label{e=e_c}For $R \in \Com$, $E(R) \cong W(R)$.
\end{theorem}
\begin{proof} Let $ R \in \Com$. We denote by $E_C: \Com \to \Ab$ the pre-Witt functor defined in \cite[Section 2]{ps}. Note that in \cite{ps}, it was denoted by $E :\Com \to \Ab$, however to avoid conflict with  $E:\Rings \to \Ab$ (see Definition \ref{def:e}), we denote it by $E_C$ instead. We know that $E_C(R) \cong W(R)$ (see \cite[Page 561]{cd}), hence it suffices to show that $E(R)\cong E_C(R)$ whenever $R \in \Com$. We will prove the result in following steps. \\

\noindent {\bf Step 1:} In this step we will prove that there exists a group homomorphism say $\tilde{f}: X(\Z\{R\}) \to E_C(R)$ which factors via $E(R)$. Consider following two presentations of $R$.\\
\vspace{3mm}
 $0 \to I \to \Z\{R\} \stackrel{\pi_1}{\longrightarrow} R \to 0 \text{ \ and \ }$ 
 $0 \to K \to \Z[R] \stackrel{\pi_2}{\longrightarrow} R \to 0$. \\
 Since $R$ is commutative, $\pi_1$ factors via $\Z[R]$. Thus,  $\exists \ \tilde{f}:  \Z\{R\} \to \Z[R]$ such that $\pi_1 = \pi_2 \circ \tilde{f}$ and we get the following commutative diagram.
\[\begin{tikzcd}
0 \arrow[r] & I \arrow[r] \arrow[d]  & \Z\{R\} \arrow[r,"\pi_1"]\arrow[d, "\tilde{f}"]           & R \arrow[r]\arrow[d, "\text{id}"]          & 0 \\
0 \arrow[r] & K \arrow[r]           & \Z[R] \arrow[r,"\pi_2"]           & R \arrow[r]           & 0 
\end{tikzcd}\]
Applying the functor $X$, we get the induced morphism
$\tilde{f}: X(\Z\{R\}) \to E_C(R)$ given by $\tilde{f}(V^n\ab{\alpha}) : = V^n\ab{\tilde{f}(\alpha)}$. \\

\noindent Enough to show that $\tilde{f}(\tilde{X_I}) \subseteq X_K$. Since $\tilde{f}(X_I) \subseteq X_K$, it remains to show that $\tilde{f}(\Sigma_I) \in X_K$. \\ 

\noindent Suppose $\alpha \in \Sigma_I$, i.e. $\exists \ S \in \Rings$ such that $S$ and $\bar{S}$ are p-torsion free  and $\exists \ S \xrightarrow{\phi} R$. Moreover,  $\tilde{\phi}(\beta)=\alpha$ for some $\beta \in X_J^{\sat}$. As $R$ is commutative ring, $\phi$ factors through $\bar{S}$, i.e. $\exists \ \psi: \bar{S} \to R$ such that $\psi \circ \pi = \phi$. Thus, we need to consider the extended diagram in $\Rings$.

\[\begin{tikzcd}
               & 0 \arrow[rr] &                                          & J \arrow[rr] \arrow[d] &                                                            & \Z\{S\} \arrow[rr] \arrow[d, "\tilde{\pi}"{pos=0.3}] &                                                      & S \arrow[rr] \arrow[d,"\pi"] &   & 0 \\
             & 0 \arrow[rr] &                                          & \bar{J} \arrow[rr] \arrow[dd] \arrow[ld, dashed] &                                                            & \Z\{\bar{S}\} \arrow[rr] \arrow[dd, "\tilde{\psi}"{pos=0.3}] \arrow[ld, dashed, "\pi' "] &                                                      & \bar{S} \arrow[rr] \arrow[ld, dashed, "id"] \arrow[dd,"\psi"{pos=0.3}] &   & 0 \\
0 \arrow[rr] &              & M \arrow[rr, dashed] \arrow[rdd, dashed] &                                            & {\Z[\bar{S}]} \arrow[rr, dashed] \arrow[rdd, dashed, "\chi"{pos=0.3}] &                                                   & \bar{S} \arrow[rr, dashed] \arrow[rdd, dashed] &                                            & 0 &   \\
             & 0 \arrow[rr] &                                          & I \arrow[rr] \arrow[d] &               & \Z\{R\} \arrow[rr,"\pi_1"] \arrow[d, "\tilde{f}"{pos=0.3}] &                                                 & R \arrow[rr] \arrow[d, "id"] &   & 0 \\
             & 0 \arrow[rr] &                 & K \arrow[rr]  &                             & {\Z[R]} \arrow[rr,"\pi_2"]  &               & R \arrow[rr]  &            &  0
\end{tikzcd}\]

\noindent Considering the diagram after applying the functor $X$ (see Diagram(2) in Definition \ref{def:e}), it  is easy to see that $\tilde{\pi}(X_J^{\sat}) \subseteq X_{\bar{J}}^{\sat}$, $\tilde{\psi}(X_{\bar{J}}^{\sat}) \subseteq X_I^{\sat} $ and $\tilde{f}(X_I^{\sat}) \subseteq X_K^{\sat}$.  Also note that $\Z[R]$ being commutative, $\tilde{f} \circ \tilde{\psi}$ factors via $\Z[\bar{S}]$. Since $\Z[\bar{S}]$ (resp. $\Z[R]$) is a commutative $p$-torsion free ring, by Corollary \cite[Cor. 2.10]{cd}, $X_M^{\sat} = X_M$ (resp. $X_K^{\sat} = X_K$). Thus,
$$\tilde{f}(\alpha) = \tilde{f} \circ \tilde{\phi}(\beta) = \tilde{f} \circ \tilde{\psi}\circ \tilde{\pi}(\beta) = \chi \circ \pi' \circ \tilde{\pi}(\beta) \in X_K$$
Therefore there exists a well defined group homomorphism,

$$\tilde{f}: E(R):=\frac{X(\Z\ac{R})}{\tilde{X_I} }\longrightarrow  \frac{X(\Z[R])}{X_K} =: E_c(R).$$

\noindent {\bf Step 2:} We now will establish a group homomorphism say $\tilde{g} : E_C(R) \to E(R)$ such that $\tilde{g} \circ \tilde{f} = id_{E(R)}$ and $\tilde{f} \circ \tilde{g} = id_{E_C(R)}$.\\

Consider the short exact sequence and the following diagram. 
\[\begin{tikzcd}
     &                 &    0\arrow[d]   &  0 \arrow[d] &         & \\
    &                 &  I \arrow[d] \arrow[r] &   K\arrow[d]       &               & \\
0 \arrow[r] & N \arrow[r] & \Z\{R\} \arrow[r]\arrow[d, "\pi_1"]  & \Z[R] \arrow[r]\arrow[d, "\pi_2"] \arrow[l, "\tilde{g}", dashed, bend right=49] & 0 \\
 & 0  \arrow[r]    & R \arrow[r, "id"]\arrow[d] & R \arrow[r]\arrow[d] & 0 \\
  &  & 0 & 0 & 
\end{tikzcd}\]
Since $\Z[R]$ is a free commutative ring, there exists a set theoretic section say $\tilde{g}: \Z[R] \to \Z\{R\}$, which we will now describe. We have set theoretic sections $R \stackrel{[\cdot]_1}{\longrightarrow} \Z\{R\}$ and $R \stackrel{[\cdot]_2}{\longrightarrow} \Z[R]$. We fix an ordering on the set $\text{Im}(R) \subset \Z\{R\}$ and define a set map $\tilde{g} : \Z[R] \to \Z\{R\}$  by 
$$\sum_{j=1}^m a_{i_1i_2\cdots i_j}[r_{i_1}]_2^{n_1}\cdots [r_{i_j}]_2^{n_j} \mapsto \sum_{j=1}^m a_{i_1i_2\cdots i_j}[r_{i_1}]_1^{n_1}\cdots [r_{i_j}]_1^{n_j}$$
 
 \noindent We get a induced well defined set map on the quotients $\tilde{g}: X(\Z[R])/X_K \to X(\Z\{R\})/X_N$. This is also an additive group homomorphism. Note that $N \subseteq I$, thus $X_N \subseteq X_I \subseteq \tilde{X_I}$. Thus we get a group homomorphism
  $$\tilde{g}: E_C(R) := X(\Z[R])/X_K \longrightarrow X(\Z\{R\})/X_I \twoheadrightarrow X(\Z\{R\})/\tilde{X_I} =: E(R)$$

\noindent Now for $\alpha \in X(\Z[R])$, $\tilde{f} \circ \tilde{g}(V^n\ab{\alpha})= \tilde{f}V^n(\ab{\tilde{g}(\alpha)}) = V^n\ab{\tilde{f}\circ \tilde{g}(\alpha)} \equiv V^n\ab{\alpha} (\text{mod } X_K )$. \\ 
\noindent Therefore,  $\tilde{f} \circ \tilde{g} = id_{E_C(R)}$. Similarly, for $\beta \in X(\Z\{R\})$, $\tilde{g} \circ \tilde{f}(V^n\ab{\beta}) = V^n\ab{\tilde{g}\circ\tilde{f}(\beta)}$ and $V^n\ab{\tilde{g}\circ\tilde{f}(\beta)} \equiv V^n\ab{\beta}(\text{mod } X_I )$.
\end{proof}

\noindent This proves Theorem \ref{main}(1). The proof of Theorem \ref{main}(2) is postponed to the Section 4. In the next section, assuming Conjecture \ref{zind}, we show that $E$ is a universal pre-Witt functor if $p\neq 2$. 
\section{Proof of Theorem \ref{main2}}
 \noindent As part of the proof of Theorem \ref{main2}, we construct of a functor $C: \Rings \to \mathsf{Ab}$ which is obtained by enforcing the properties \ref{properties} (1)-(3). It is less straightforward to see that $C$ satisfies \ref{properties}(4). We introduce a temporary notion of a weak pre-Witt functor (see Definition \ref{def:weakpre}) to show that $C$ is a pre-Witt functor (see Theorem \ref{universalweakprewitt}). We will in fact show that $C$ is a universal pre-Witt functor. Essentially, this section has a lot of conceptual overlap with \cite[Section 3]{ps}. We will follow the arguments used in \cite[Section 3]{ps}.

\begin{definition}[The functor $C$]\label{def:C}
For any associative ring $R$ and a prime $p\neq 2$, we first define a group $G(R)$ as follows. Let 
\begin{enumerate}
\item $\tilde{G}(R)$ be the free abelian group generated by symbols $\{ V^n_r \ | \ n\in \N_0 , \ 0\neq r\in R\}$. Define $V^n_0:=0$ for all $n\geq0$. 
\item A set map $R \xrightarrow{\ab{\ } } \tilde{G}(R)$ given by $\ab{r} := V^0_r$.
\item A homomorphism $\tilde{G}(R) \xrightarrow{V} \tilde{G}(R)$ defined by $V(V^n_r):= V^{n+1}_r$.
\item $H(R)\subset \tilde{G}(R)$ be the subgroup generated by the set: 
$$  \Union_n\left(\{(V^n_{(x+y)^p} - p V^{n-1}_{x+y}) - (V^n_{x^p} - p V^{n-1}_{x}) - (V^n_{y^p} - p V^{n-1}_{y})\ | \ x, y \in R \} \Union \{V^n_r+V^n_{-r} \ | \ r \in R\}\right)$$ 
\item $\tilde{H}(R)$ be the $p$-saturation of $H(R)$, i.e.
$ \tilde{H}(R):= \{ \alpha \in \tilde{G}(R) \ | \ p^\ell \alpha \in H(R) \ \text{for some } \ell >0\}$.
\item Let $G^0(R)$ denote the completion of $\tilde{G}(R)/\tilde{H}(R)$ by the $V$-filtration. Let $G(R)$ be the quotient of $G^0(R)$ modulo the closed subgroup generated by $p$-power torsion elements. Note that the construction of $G(R)$ is functorial in $R$.\\
\end{enumerate}
Consider the free presentation $0 \to I \to A \to R \to 0 $ where $A:=\Z\{R\}$. For $x \in R$ we denote the corresponding variable in $A$ by $[x]$. Let $G_I(A) \subset G(A)$ is the closed subgroup generated by $$\{V^n_a-V^n_b \ | \ n\geq 0,  \ a,b\in A \ \text{ such that } a-b\in I\}.$$

\noindent Now define $$C(R):= \frac{G(A)}{\tilde{G_I}}$$ where $\tilde{G_I}$ is subgroup generated by $G_I$ and $\Sigma_I$ where $\Sigma_I$ and $G_I$ are defined similar to the definition of $\Sigma_I$ and $X_I$ as in Definition \ref{def:e}. 
\end{definition} 

\begin{remark} We have a well defined group homomorphism $V: C(R)\to C(R)$ and a well defined map of sets $R \xrightarrow{\ab{ \ }} C(R)$ induced by the corresponding maps on $G(\Z[R])$. Any ring homomorphism $f: R \to S,$ gives a group homomorphism $G(R) \to G(S), V^n_r \mapsto V^n_{f(r)}$ which is compatible with $V$ and hence a group homomorphism $C(f): C(R) \to C(S)$. 
\end{remark}

Let $R\in \Rings$. In the definition of $C(R)$ we chose the canonical free presentation $0 \to I \to \Z\{R\} \to R\to 0$. However if  $ 0 \to K \to B \to R \to 0$ is any other free presentation we can also define the group 
$$C_B(R) := \frac{G(B)}{\tilde{G_K}}$$
Following the arguments as in Lemma \ref{ind:e} and Lemma \ref{EisX} we can prove the following results.

\begin{lemma}\label{ind:C} $C(R) \cong C_B(R)$.
\end{lemma}

\begin{corollary}\label{ceqg}
If $A=\Z\{S\}$ is a free algebra, then $C(A) \cong G(A)$.
\end{corollary}

\noindent We will eventually prove the following result.
\begin{theorem}
For $p \neq 2$, the functor $C$ is a universal pre-Witt functor.
\end{theorem}

\noindent Unfortunately, it is not straightforward to verify that $C$ satisfies property~\ref{properties}(4) and hence is a pre-Witt functor. To overcome this difficulty, we introduce below the notion of a \emph{weak pre-Witt functor}, whose role is intended to be temporary. First, we show that $C$ is a weak pre-Witt functor.  We then prove that, for any pre-Witt functor $F$, there exists a natural transformation $C \to F$. Finally, we show that $C \cong E$, from which it follows that $C$ is a universal pre-Witt functor.
\begin{definition}[weak pre-Witt functor]\label{def:weakpre} A functor $F: \Rings \to  \mathsf{Ab}$ is said to be a weak pre-Witt functor if it satisfies the properties (1), (2) and (3) of \ref{properties} and the following property, which is a weaker version of  \ref{properties}(4): \\

$(4')$ $A$ is a free polynomial ring in $\Rings$ then $F(A)$ is $p$-torsion free
\end{definition} 
\vspace{1mm}
\noindent Any pre-Witt functor is a weak pre-Witt functor. In particular, the functor $E$ (see Definition \ref{def:e}), the Witt functor $W_H$ defined in \cite{h1}, \cite{h2} are weak pre-Witt functors. 
 
\begin{theorem}\label{universalweakprewitt} Let $p\neq 2$.  The functor $C: \Rings \to \Ab$ as defined in \ref{def:C}, is a weak pre-Witt functor.  Moreover, given any pre-Witt functor $F$, there exists a natural transformation $C \to F$. In particular, there is a natural transformation $C\xrightarrow{\eta} E$ which is compatible with $V, \ab{ \ }$.
\end{theorem}

\begin{proof} It follows from the construction of $C$ that, $C$ satisfies properties (1), (2) and (3) of Definition \ref{properties}. The Corollary \ref{ceqg} implies that $C$ also satisfies the property $(4')$ of the Definition \ref{def:weakpre}. Hence $C$ is a weak pre-Witt functor. We will now prove that given any pre-Witt functor $F$ there exists a natural transformation $\eta: C \to F$. The idea is to first establish a group homomorphism in the case of a free algebra and then extend the result to any $R \in \Rings$. We will prove this in the following steps. \\

\noindent Step 1: We will show that, for a free non-commutative polynomial ring $A$ there exists a group homomorphism $C(A) \to F(A)$. Consider the presentation $$0 \to 0 \to A \xrightarrow{id} A \to 0$$ We define the homomorphism, 

$$\eta : \tilde{G}(A) \to F(A) \ \text{given by} \  V^n_{a}  \mapsto V^n\ab{a} \ \forall \  a \in A, \ n \geq 0.$$ 

\vspace{1mm}
\noindent Since $F$ satisfies the properties (1)-(2) of Definition \ref{def:weakpre}, $\eta(\tilde{H}(A))=0$ in $F(A)$. So, we get a homomorphism,

 $$\tilde{\eta}:\tilde{G}(A)/\tilde{H}(A) \to F(A).$$ 
 
\vspace{1mm}
\noindent As $\eta$ is compatible with $V$, taking $V$-completion, we get the induced group homomorphism, denoted as $\eta_A: G(A) \to F(A)$. \\

\noindent{Step 2:} To show that for $R \in \Rings$, there exists a group homomorphism $C(R) \to F(R)$.\\

\noindent Consider the presentation $0 \to I \to \Z\{R\} \stackrel{\pi_1}{\longrightarrow} R \to 0$ of $R$. By Step (1), we have a group homomorphism $F(\pi_1) \circ \eta_{\Z\{R\}} : G(\Z\{R\}) \to F(R)$. \\
\vspace{2mm}

\noindent It is easy to see that for $a, b \in \Z\{R\}$ such that $a-b \in I$, $G(\pi_1)(V^n_a-V^n_b)= V^n\ab{\pi_1(a)}-V^n\ab{\pi_1(b)} = 0$, thus $F(\pi_1)\circ\eta_{\Z\{R\}}(G_I) = 0$. Hence there exists a group homomorphism $\eta: G(\Z\{R\})/G_I \to F(R)$. It remains to show that $F(\pi_1) \circ \eta_{\Z\{R\}}(\tilde{G_I}) = 0$ \ \\

\noindent Let $S \in \Rings$ such that $S, \bar{S}$ are $p$-torsion free and $\phi: S \to R$ be a ring homomorphism. We now recall here the definition of $\Sigma$ and the diagram (1) as in the given in Definition \ref{def:e}. 
\[
\begin{tikzcd}[column sep=1.5em, row sep=1.5em]
0 \arrow[r] & J \arrow[d, dashed] \arrow[r] & \Z\{S\} \arrow[d, dashed,"\tilde{\phi}"] \arrow[r, "\pi_2"] & S \arrow[d, "\phi"] \arrow[r] & 0 \\
0 \arrow[r] & I \arrow[r] & \Z\{R\} \arrow[r, "\pi_1"] & R \arrow[r] & 0
\end{tikzcd}
\tag{1}
\] 
\noindent Consider the map $f: G(\Z\{S\}) \to G(\Z\{S\})/G_J$. Then 
$$\Sigma := \{ \alpha \in G(\Z\{R\}) \mid \alpha = G(\tilde{\phi})(\beta), \ \beta \ \text{is a p-torsion in \ }  G(\Z\{S\})/G_J \}$$

\noindent We have the following diagram 
\begin{center}
\begin{tikzcd}
&
G(\Z\{S \})
\ar{dl}[swap, sloped, near start]{\eta_{\Z\{S\}}}
\ar{rr}{f}
\ar[]{dd}[near start]{\tiny{G(\tilde{\phi})}}
& & G(\Z\{S\})/G_J 
\ar{dd}{}
\ar{dl}[swap, sloped, near start]{g}
\\
F(\Z\{S\}) 
\ar[crossing over]{rr}[swap, near start]{F(\pi_2)}
\ar{dd}[swap]{F(\tilde{\phi})}
& & F(S)
\ar[]{dd}[near start]{F(\phi)}
\\
&
G(\Z\{R\})
\ar[near start]{rr}{}
\ar[sloped]{dl}{\eta_{\Z\{R\}}}
& &  G(\Z\{R\})/G_I
\arrow[dl, densely dotted]
\\
F(\Z\{R\})
\ar[rr, "F(\pi_1)"]
& & F(R)
\ar[crossing over, leftarrow, near start]{uu}{}
\end{tikzcd}
\end{center}

\noindent Let $\alpha \in \Sigma$. Then $\alpha = G(\tilde{\phi})(\beta)$ for some $\beta \in G_J^{\sat}$. Thus $p^{\ell}\beta \in G_J$ for some $\ell \geq 0$. In other words $f(\beta)$ is a $p$-torsion element in $G(\Z\{S\})/G_J$. Since $\eta_{\Z\{S\}}(G_J) \subseteq F_J$, we have the induced group homomorphism denoted by $g: G(\Z\{S\})/G_J \to F(S)$, which takes $f(\beta)$ to a $p$-torsion element in $F(S)$. Since $F$ is a pre-Witt functor and $S, \bar{S}$ are $p$-torsion free,  $F(S)$ is $p$-torsion free. Thus $g \circ f(\beta) = 0$. Thus,
$$(F(\phi) \circ g \circ f)(\beta) = (F(\phi)\circ F(\pi_2)\circ \eta_{\Z\{S\}})(\beta) = 0$$
Since $F$ is functor, we have the commutativity of the front square and thus,
$$(F(\phi) \circ F(\pi_2) \circ \eta_{\Z\{S\}})(\beta) = (F(\pi_1) \circ F(\tilde{\phi}) \circ \eta_{\Z\{S\}})(\beta) = 0  $$

\noindent By step (1), the left outermost square commutes. Hence,
$$(F(\pi_1) \circ F(\tilde{\phi}) \circ \eta_{\Z\{S\}})(\beta) = (F(\pi_1) \circ \eta_{\Z\{R\}} \circ G(\tilde{\phi})(\beta) = 0 $$
By definition $\alpha = G(\tilde{\phi})(\beta)$, hence 
$$F(\pi_1) \circ \eta_{\Z\{R\}}(\beta) = 0 $$
Since $\eta_{\Z\{R\}}(G_I) \subseteq F_I$, the morphism $F(\pi_1) \circ \eta_{\Z\{R\}}$ already factors through $G(\Z\{R\})/G_I$. Thus we have a morphism 
$$C(R) : = G(\Z\{R\})/\langle G_I, \Sigma\rangle \to F(R)$$
Thus, in particular, since $E$ is a pre-Witt functor, there exists a natural transformation $C \to E$. 
\end{proof}

In the remaining section, we show that $C$ is in fact naturally isomorphic to $E$ subject to the Conjecture \ref{zind}. This implies that  $C$ is pre-Witt and Theorem \ref{universalweakprewitt} implies that $C$ is in fact a universal pre-Witt functor. Thus, $E$ is a universal pre-Witt functor. The first step towards proving $C \cong E$ would be to show that for a free non-commutative polynomial algebra $A$, $C(A) \cong X(A)$.  

\begin{proposition}\label{C(A)isX(A)} Let $p \neq 2$, and $A=\Z\ac{S}$ be a free ring over a set $S$. The Conjecture \ref{zind} implies that $C(A) \cong X(A)$. 
\end{proposition}
\begin{proof}
The proof of this lemma is similar to the Lemma 4.3 in \cite{ps}.
\vspace{2mm}
By using the Step(1) of the proof of the Theorem \ref{universalweakprewitt},  replacing $F$ by $E$, and using the Lemma \ref{EisX}, we get an explicit epimorphism, 
$$\tilde{\eta}:  \tilde{G}(A)/\tilde{H}(A) \to \tilde{X}(A);  \  \ \ V^n_a \mapsto V^n\ab{a} .$$
 
To show that $C(A) \cong X(A)$ it is enough to show that $\tilde{\eta}$ is an isomorphism. This is because, the map $\tilde{\eta}$ is $V$-compatible and after taking the $V$-completions, we will get that $C(A) \cong X(A)$. Thus, it is enough to show that this map is injective. The injectivity is proved by using the arguments as in the step(2) of the proof of Lemma 4.3 in \cite{ps}. Note that Lemma 4.2 in \cite{ps} will be replaced by the Conjecture \ref{zind}. Hence we get an isomorphism denoted as $\eta_{A}: C(A) \to X(A)$
\end{proof}

\begin{proof}[Proof of Theorem \ref{main2}] To show that $E$ is universal pre-Witt functor, it suffices to show that $E(R)$ is naturally isomorphic to $C(R)$ for any $R \in \Rings$. Consider the presentation, $0 \to I \to A:=\Z\{R\} \to R \to 0$. By following the proof of Theorem \ref{universalweakprewitt} with $F$ replaced by $E$, we get a morphism $C(R) \to E(R)$. By using the Lemma \ref{EisX} and Corollary \ref{ceqg}, we have the following commutative diagram,
\[
\begin{tikzcd}
  G(A) \cong C(A) \arrow[r, two heads]  \arrow[d, "\eta_{A}"] &  G(A)/\tilde{G_I} \arrow[r] \arrow[d, densely dotted] & 0 \\
   X(A) \arrow[r, two heads] &  X(A)/\tilde{X_I}  \arrow[r]              & 0
\end{tikzcd}
\]
Since $\eta_{A}$ is an isomorphism, it is easy to check that $\eta_{A}(\tilde{G_I}) = \tilde{X_I}$. This implies that $\eta_{A}$ induces an isomorphism $\eta_R: C(R) \to E(R)$. 
\end{proof}

\section{Witt functors}

\noindent The goal of this section is to define the notion of a Witt functor (see Definition \ref{witt}) and describe the Witt functor $\hat{E}$. We will prove the universality of $\hat{E}$ assuming Conjecture \ref{zind}.  We also finish the proof of Theorem \ref{main} by showing the existence of a natural transformation $E \to W_H$. Note that such a transformation exists by universality of $E$ (Theorem \ref{main2}), however we would like to establish this (see Proposition \ref{ehattowh})  without assuming Conjecture \ref{zind}. \\

\noindent In the classical case, for $R \in \Com$, the set $W(R)$ is in bijection with $R^{\N_0}$. This bijection can be explicitly given by 
$$ (a_0, a_1, ...) \mapsto \sum_n V^n\ab{a_n}$$
Let $F: \Rings \to \Ab$ be a pre-Witt functor. For a tuple $(a_0, a_1,...) \in R^{\N_0}$, denote $$(a_0, a_1,...)_w  := \sum_n V^n\ab{a_n} \in F(R).$$

\subsection{Witt sum and difference polynomials :} Before defining the Witt functor, we first recall the group structure on $W_H(R)$ for $R \in Rings$ as given in \cite[1.5, 1.6]{h1}. Note that $W_H(R)$ is equipped with the Verschiebung operator  
$$V: W(R) \to W(R) \ \ \ \  V (a_0,a_1,\cdots):= (0,a_0,a_1,\cdots) $$ 
and the Teichm\"uller map 
$$ \langle \ \rangle : R \longrightarrow W(R), \ \ \ \  \langle a\rangle := ( a,0,0,\ldots)$$

\begin{remark}
 The definitions $V: W_H(R) \to W_H(R)$ and $\langle \ \rangle : R \longrightarrow W_H(R)$ differ from that of $V: E(R) \to E(R)$ and $\langle \ \rangle : R \longrightarrow E(R)$ (see Definition \ref{def:ncxa} and Remark \ref{rmk:e}). Every element of $W_H(R)$ represented (not necessarily uniquely) as $(a_0,a_1...)_w$. This is a consequence of the fact that the addition and subtraction of elements of the form $(a_0,a_1...)_w$ in the group $W_H(R)$ are given by non-commutative version of Witt polynomials as given below. 
\end{remark}
\begin{definition}[Witt sum and difference polynomials]\label{rnen}
For $R \in \Rings$, and elements \\
$(a_0,a_1,...)_w, (b_0, b_1, ...)_w \in W_H(R)$,
$$ (a_0,a_1,...)_w +  (b_0, b_1, ...)_w  := (s_0(a_0,b_0), s_1(a_0,b_0,a_1,b_1),...)$$
$$ (a_0,a_1,...)_w -  (b_0, b_1, ...)_w  := (d_0(a_0,b_0), d_1(a_0,b_0,a_1,b_1),...)$$
where $s_i$ and $d_i$ are non-commutative analogues of classical Witt sum and difference polynomials as defined in \cite[1.4.1]{h1}. We use the following two variable specialisations of the Witt sum and difference polynomials. For an integer $i\geq 0$
define polynomials $r_i, e_i\in \Z\ac{X,Y}$ by 
\begin{align*}
 r_i(X,Y) & := s_i(X,Y,0,0,...0,0)\\
 e_i(X,Y) & := d_i(X,Y,0,0,...0,0)
 \end{align*}
It follows that $r_0(X,Y) = X+Y$ and $e_0(X,Y)=X-Y$. In particular, we have the following relations in $W_H$ (and also in the classical case)
\begin{align*}
 \ab{x} + \ab{y} & = \ab{x+y} + \sum_{i>0} V^i \ab{r_i(x,y)} \\
 \ab{x} - \ab{y} & = \ab{x-y} + \sum_{i>0} V^i \ab{e_i(x,y)}
\end{align*}
\end{definition}
\begin{definition}[Witt functor]\label{witt}
A pre-Witt functor $F : \Rings \to \Ab$ is called a Witt functor if,
\begin{align*}
 \ab{x} + \ab{y} & = \ab{x+y} + \sum_{i>0} V^i \ab{r_i(x,y)} \\
 \ab{x} - \ab{y} & = \ab{x-y} + \sum_{i>0} V^i \ab{e_i(x,y)}
\end{align*}
\end{definition}

\begin{examples} $W_H$ is a Witt functor.
\end{examples} 

\begin{definition}[The Witt functor $\hat{E}$]\label{defn:ehat}
For a ring $R \in \Rings$, let $\sR(R) \subset E(R)$ denote the smallest closed subgroup stable under $V$  and containing all  
elements of the form 
$$ \{\ab{x}+\ab{y} - \sum_{n\geq 0} V^n \ab{r_n(x,y)} \ | \ x, y \in R\} \ \Union \ \{\ab{x}-\ab{y} - \sum_{n\geq 0} V^n \ab{e_n(x,y)} \ | \ x,y \in R\}$$
where $r_n,e_n$ are integer polynomials defined in \eqref{rnen}. Let $\hat{E}(R):= \frac{E(R)}{\sR(R)}$. Clearly $\hat{E}$ is a Witt functor, since the required relations have been forced by the definition. 
\end{definition}
\noindent The following result is an immediate consequence of the above definition of $\hat{E}$.

\begin{lemma}\label{everyeleofEhat}
Every element of $\hat{E}(R)$ is of the form
\[
  \sum_{n=0}^{\infty} V^{n}\ab{a_n}.
\]
\end{lemma}

\noindent Moreover, we have the following conjectural characterisation of the functor $\hat{E}$.

\begin{theorem}
For $p \neq 2$, the Conjecture \ref{zind} implies that $\hat{E}$ is a universal Witt functor. 
\end{theorem} 
\begin{proof}
By using Conjecture \ref{zind}, we have proved that $E$ is a universal pre-Witt functor (see Theorem \ref{main2}).  $\hat{E}$ is obtained by simply going modulo the extra relations in $E$ required to be a Witt functor. From this, one can easily deduce that $\hat{E}$ is a universal Witt functor.
\end{proof}

\noindent The goal of the remaining section is to prove Theorem \ref{main}(2). We will prove the following result without using Conjecture \ref{zind} and Theorem \ref{main2}. 
 
\begin{proposition}\label{ehattowh}
There is a unique natural transformation of pre-Witt functors $E \to W_H$ and of Witt functors $\hat{E} \to W_H$. 
\end{proposition}
We will need the following results in order to the prove Proposition \ref{ehattowh}.
\begin{lemma}\label{ehatrelation}
Let $A=\Z\ac{S}$ be a free algebra. Then there is a natural isomorphism $\hat{E}(A) \cong \frac{X(A)}{R(A)}$ where
$R(A)\subset X(A)$ is the smallest $V$-stable closed subgroup containing $$ \{\ab{x}+\ab{y} - \sum_{n \geq 0} V^n \ab{r_n(x,y)} \ | \ x, y \in A\} \union \{\ab{x}-\ab{y} - \sum_{n \geq 0}V^n \ab{e_n(x,y)} \ | \ x, y \in A\}$$
\end{lemma}
\begin{proof}
 By Corollary \ref{EisX}, we know that $E(A) \cong X(A)$. Since this isomorphism is compatible with $V$ and $\ab{ \ }$, $\sR(A)$ is mapped to $R(A)$
\end{proof}
\noindent We now observe the following connection between  $\hat{E}(A)$ and $W_H(A)$.
 
\begin{lemma}\label{imageehat} 
Let $A=\Z\ac{S}$ be a free algebra. Then the image of the natural map $\hat{E}(A) \to (\frac{A}{[A,A]})^{\N_0}$ is canonically isomorphic to $W_H(A)$. This gives a canonical surjection $\hat{E}(A) \to W_H(A)$ sending $V^n\ab{a} \mapsto V^n\ab{a}$. 
\end{lemma}

\begin{proof}
Let $\eta : X(A) \to \left(\frac{A}{[A,A]}\right)^{\N_0}$ be the map given by the composition $$X(A) \hookrightarrow A^{\N_0} \to \left(\frac{A}{[A,A]}\right)^{\N_0}.$$
 We will show that $R(A) \subseteq {\text Ker}(\eta)$.\\
\noindent We know that the ghost map $\omega: W_H(A) \to  \left(\frac{A}{[A,A]}\right)^{\N_0}$ is injective (see \cite[page 2]{h2}). Hence, the image  $\omega(W_H(A))$ is canonically isomorphic to $W_H(A)$. We consider the following diagram:
\begin{center}
\begin{tikzcd}
  X(A) \arrow[swap]{dr}{\phi} \arrow[r, "\eta"] & \left(\frac{A}{[A,A]}\right)^{\N_0} \\
     & W_H(A) \arrow[u, "\omega"]
\end{tikzcd}
\end{center}
Let $\phi: X(A) \to W_H(A)$ be  a map that sends an element $V^n\ab{a}$ to $V^n\ab{a}$. We also recall that $W_H(A)$ is generated by elements of the form $\{V^n\ab{a_n}\mid a_n \in A, n \in \N_0\}$ by \cite[Lemma 3.3]{hp},  hence, $\phi$ is a surjective map. Also, 
$$\omega(\sum_n V^n\ab{a_n}) = \omega(a_0,a_1,\cdots) = (\overline{a_0},{\overline{a_0}}^p+{\overline{a_1}}, \cdots)$$
It is easy to verify that,
$$\eta(\sum_n V^n\ab{a_n}) = (\overline{a_0},{\overline{a_0}}^p+{\overline{a_1}}, \cdots).$$
Hence, the image of the natural map $X(A) \to (\frac{A}{[A,A]})^{\N_0}$ is canonically isomorphic to $W_H(A)$. Note that the `relations' defining $R(A)$ (see Definition \ref{defn:ehat}) are mapped to zero in $W_H(A)$ and hence are mapped to zero in 
$(\frac{A}{[A,A]})^{\N_0}$. Therefore $\eta$ induces a map 
$$ \frac{X(A)}{R(A)} \cong \hat{E}(A) \xrightarrow{\hat{\eta}} \left(\frac{A}{[A,A]}\right)^{\N_0}$$
and a canonical surjection
$$ \hat{E}(A) \to W_H(A)$$
\end{proof}

\begin{proof}[Proof of Proposition \ref{ehattowh}]  For $R \in \Rings$, consider the presentation $0 \to I \to A:=\Z\ac{R} \to R \to 0$. Lemma \ref{EisX} and the proof of Lemma \ref{imageehat} gives a natural morphism $\eta_A: E(A) \cong X(A) \to W_H(A)$ where $A$ is any free non-commutative algebra. Thus, there exists a morphism $\eta_{\Z\{R\}}: E(\Z\{R\}) \to W_H(\Z\{R\})$ given by, $V^n\ab{a} \mapsto V^n\ab{a}$. In order to show that there exists a morphism $\eta_R: E(R) \to W_H(R)$ we need to show that $\eta_{\Z\{R\}}(\tilde{X_I}) = 0$. By using the Step(2) of the proof of the Theorem \ref{universalweakprewitt}, and replacing $C$ (resp. $F$) by $E$ ( resp. $W_H$) we get that there exists a morphism $\eta_R :  E(R) \to W_H(R)$ between the pre-Witt functors. It is now straightforward to observe that the `relations' defining $\hat{E}(R)$ are mapped to zero in $W_H(R)$. Hence we have a unique natural transformation $\hat{E} \to W_H$ of Witt functors. 
\end{proof}

\begin{proof}[Proof of Theorem \ref{main}]
Follows from  Theorem \ref{e:pre-witt}, Theorem \ref{e=e_c} and  Proposition \ref{ehattowh}. 
\end{proof} 

\section{Evidence for Conjecture \ref{zind}}\label{evidence}

The goal of this section is to discuss the computational evidence for Conjecture \ref{zind}, in addition to the fact that the similar statement holds for commutative polynomial rings (\cite[Lemma 4.2]{ps}).\\

\noindent Let $A = \Z\{X,Y\}$ and $p=3$. For non-zero elements $f_1, f_2, f_3\in A$ which have distinct moduli, we would like to check that $\{\ab{f_i} \}$ is $\Z$-linearly independent subset of $X(A)$. This is clearly true in either of the following cases:
\begin{enumerate}
\item $f_1, f_2, f_3$ are $\Z$-linearly independent.
\item The images of $f_i$ in the commutative polynomial ring $\Z[X,Y]$ have distinct moduli (\cite[Lemma 4.2]{ps}). 
\end{enumerate}

\noindent In order to find computational evidence for Conjecture \ref{zind}, we generated samples $f_1,f_2,f_3$ in $A$ which do not satisfy the above two conditions. For a few thousand of such samples (of degree $\leq 3$) we have verified the conjecture using SAGE MATH 


\end{document}